\documentclass[fleqn,12pt,twoside]{article}

\usepackage[headings]{espcrc1}
\readRCS
$Id: espcrc1.tex,v 1.2 2004/02/24 11:22:11 spepping Exp $
\usepackage{graphicx}
\usepackage[figuresright]{rotating}

\newtheorem{theorem}{Theorem}

\newtheorem{question}{Question}

\newtheorem{corollary}[theorem]{Corollary}

\newtheorem{lemma}[theorem]{Lemma}

\newenvironment{proof}[1][Proof.]{\begin{trivlist}
\item[\hskip \labelsep {\bfseries #1}]}{\end{trivlist}}

\newcommand{\AmS}{{\protect\the\textfont2
  A\kern-.1667em\lower.5ex\hbox{M}\kern-.125emS}}

\usepackage{amsmath}
\usepackage{amsfonts}
\title{On interval edge-colorings of outerplanar graphs}

\author{Petros A. Petrosyan\address[MCSD]{Department of Informatics and Applied Mathematics,\\
Yerevan State University, 0025, Armenia}%
\address{Institute for Informatics and Automation Problems,\\
National Academy of Sciences, 0014, Armenia}%
\thanks{email: pet\_petros@\{ipia.sci.am, ysu.am, yahoo.com\}}}

\runtitle{On interval edge-colorings of outerplanar
graphs}\runauthor{Petros A. Petrosyan}

\begin{document}

\maketitle

\begin{abstract}
An edge-coloring of a graph $G$ with colors $1,\ldots,t$ is called
an interval $t$-coloring if all colors are used, and the colors of
edges incident to any vertex of $G$ are distinct and form an
interval of integers. A graph $G$ is interval colorable if it has an
interval $t$-coloring for some positive integer $t$. For an interval
colorable graph $G$, the least value of $t$ for which $G$ has an
interval $t$-coloring is denoted by $w(G)$. A graph $G$ is
outerplanar if it can be embedded in the plane so that all its
vertices lie on the same (unbounded) face. In this paper we show
that if $G$ is a $2$-connected outerplanar graph with $\Delta(G)=3$,
then $G$ is interval colorable and
\begin{center}
$w(G)=\left\{
\begin{tabular}{ll}
$3$, & if $\vert V(G)\vert$ is even, \\
$4$, & if $\vert V(G)\vert$ is odd. \\
\end{tabular}%
\right.$
\end{center}

We also give a negative answer to the question of Axenovich on the
outerplanar triangulations.\\


\end{abstract}

\section{Introduction}\

In this paper we consider graphs which are finite, undirected, and
have no loops or multiple edges. Let $V(G)$ and $E(G)$ denote the
sets of vertices and edges of a graph $G$, respectively. The degree
of a vertex $v\in V(G)$ is denoted by $d_{G}(v)$, the maximum degree
of $G$ by $\Delta (G)$, and the chromatic index of $G$ by
$\chi^{\prime}\left(G\right)$. A graph $G$ is outerplanar if it can
be embedded in the plane so that all its vertices lie on the same
(unbounded) face. An outerplanar triangulation is an outerplanar
graph in which every bounded face is a triangle. An edge of the
outerplanar graph is internal if it does not belong to unbounded
face. A separating triangle of the outerplanar graph is a triangular
face in which every edge is internal. The terms and concepts that we
do not define can be found in \cite{b10}.

A proper edge-coloring of a graph $G$ is a coloring of the edges of
$G$ such that no two adjacent edges receive the same color. If
$\alpha $ is a proper edge-coloring of $G$ and $v\in V(G)$, then
$S\left(v,\alpha \right)$ denotes the set of colors of edges
incident to $v$. A proper edge-coloring of a graph $G$ with colors
$1,\ldots ,t$ is called an interval $t$-coloring if all colors are
used, and for any vertex $v$ of $G$, the set $S\left(v,\alpha
\right)$ is an interval of integers. A graph $G$ is interval
colorable if it has an interval $t$-coloring for some positive
integer $t$. The set of all interval colorable graphs is denoted by
$\mathfrak{N}$. For a graph $G\in \mathfrak{N}$, the least value of
$t$ for which $G$ has an interval $t$-coloring is denoted by
$w\left(G\right)$.

The concept of interval coloring of graphs was introduced by
Asratian and Kamalian \cite{b1,b2}. In \cite{b1,b2}, they proved
that if $G$ is interval colorable, then $\chi^{\prime
}\left(G\right)=\Delta(G)$. They also showed that if a triangle-free
graph $G$ has an interval $t$-coloring, then $t\leq \left\vert
V(G)\right\vert -1$. In \cite{b6}, Kamalian investigated interval
colorings of complete bipartite graphs and trees. In particular, he
proved that the complete bipartite graph $K_{m,n}$ has an interval
$t$-coloring if and only if $m+n-\gcd(m,n)\leq t\leq m+n-1$, where
$\gcd(m,n)$ is the greatest common divisor of $m$ and $n$. In
\cite{b7}, Petrosyan investigated interval colorings of complete
graphs and $n$-dimensional cubes. In particular, he proved that if
$n\leq t\leq \frac{n\left(n+1\right)}{2}$, then the $n$-dimensional
cube $Q_{n}$ has an interval $t$-coloring. Recently, Petrosyan,
Khachatrian and Tananyan \cite{b8} showed that the $n$-dimensional
cube $Q_{n}$ has an interval $t$-coloring if and only if $n\leq
t\leq \frac{n\left(n+1\right)}{2}$. In \cite{b9}, Sevast'janov
proved that it is an $NP$-complete problem to decide whether a
bipartite graph has an interval coloring or not.

Proper edge-colorings of outerplanar graphs were investigated by
Fiorini in \cite{b4}. In \cite{b4}, he proved that if $G$ is an
outerplanar graph, then $\chi^{\prime}\left(G\right)=\Delta(G)$ if
and only if $G$ is not an odd cycle. Interval edge-colorings of
outerplanar graphs were first considered by Axenovich in \cite{b3}.
In \cite{b3}, she proved that all outerplanar triangulations with
more than three vertices and without separating triangles are
interval colorable. Later, interval edge-colorings of outerplanar
graphs were investigated by Giaro and Kubale in \cite{b5}, where
they proved that all outerplanar bipartite graphs are interval
colorable.

In the present paper we show that if $G$ is a $2$-connected
outerplanar graph with $\Delta(G)=3$, then $G$ is interval colorable
and
\begin{center}
$w(G)=\left\{
\begin{tabular}{ll}
$3$, & if $\vert V(G)\vert$ is even, \\
$4$, & if $\vert V(G)\vert$ is odd. \\
\end{tabular}%
\right.$
\end{center}

We also give a negative answer to the question of Axenovich on the
outerplanar triangulations.

\bigskip

\section{Interval edge-colorings of subcubic outerplanar graphs}\

First we need the following lemma which was proved in \cite{b4}.

\begin{lemma}
\label{mylemma} If $G$ is a $2$-connected outerplanar graph with
$\Delta(G)=3$, then $G$ has either (1) $u$ and $v$ adjacent vertices
such that $d_{G}(u)=d_{G}(v)=2$, or (2) $u,v$ and $w$ mutually
adjacent vertices such that $d_{G}(u)=d_{G}(w)=3$ and $d_{G}(v)=2$.
\end{lemma}

Now we can prove our first result.

\begin{theorem}
\label{mytheorem1} If $G$ is a $2$-connected outerplanar graph $G$
with $\Delta(G)\leq 3$ and $G$ is not an odd cycle, then $G\in
\mathfrak{N}$ and $w(G)\leq 4$.
\end{theorem}
\begin{proof}
For the proof, it suffices to show that if $G$ is a $2$-connected
outerplanar graph $G$ with $\Delta(G)\leq 3$ and $G$ is not an odd
cycle, then $G$ has an interval coloring with no more than four
colors.

We show it by induction on $\vert E(G)\vert$. The statement is
trivial for the case $\vert E(G)\vert \leq 5$. Assume that $\vert
E(G)\vert \geq 6$, and the statement is true for all $2$-connected
outerplanar graphs $G^{\prime}$ with $\Delta(G^{\prime})\leq 3$
which are not odd cycles and $\vert E(G^{\prime})\vert<\vert
E(G)\vert$.

Let us consider a $2$-connected outerplanar graph $G$ with
$\Delta(G)\leq 3$ which is not an odd cycle. If $\Delta(G)=2$, then
$G\in \mathfrak{N}$ and $w(G)\leq 2$. Now suppose that
$\Delta(G)=3$. By Lemma \ref{mylemma}, $G$ has either $u$ and $v$
adjacent vertices such that $d_{G}(u)=d_{G}(v)=2$, or $u,v$ and $w$
mutually adjacent vertices such that $d_{G}(u)=d_{G}(w)=3$ and
$d_{G}(v)=2$. We consider two cases.

Case 1: $uv\in E(G)$ and $d_{G}(u)=d_{G}(v)=2$.

Clearly, in this case there are vertices $x,y$ ($x\neq y$) in $G$
such that $ux\in E(G)$ and $vy\in E(G)$.

Case 1.1: $xy\notin E(G)$.

In this case let us consider a $2$-connected outerplanar graph
$G^{\prime}=(G-u-v)+xy$. By induction hypothesis, $G^{\prime}$ has
an interval coloring $\alpha$ with no more than four colors.

If $\alpha(xy)=1$, then we delete the edge $xy$ and color the edges
$ux$ and $vy$ with color $1$ and the edge $uv$ with color $2$. If
$\alpha(xy)\geq 2$, then we delete the edge $xy$ and color the edges
$ux$ and $vy$ with color $\alpha(xy)$ and the edge $uv$ with color
$\alpha(xy)-1$. It is not difficult to see that the obtained
coloring is an interval coloring of the graph $G$ with no more than
four colors.

Case 1.2: $xy\in E(G)$.

In this case let us consider a $2$-connected outerplanar graph
$G^{\prime}=G-u-v$. If $G^{\prime}$ is an odd cycle, then we color
the edges $xy$ with color $3$ and the edges of the $x,y$-path in
$G^{\prime}-xy$ alternately with colors $1$ and $2$. Next, we color
the edge $ux$ with color $2$, the edge $uv$ with color $3$ and the
edge $vy$ with color $4$. Clearly, the obtained coloring is an
interval $4$-coloring of the graph $G$.

Now we can suppose that $G^{\prime}$ is a $2$-connected outerplanar
graph with $\Delta(G^{\prime})\leq 3$ which is not an odd cycle. By
induction hypothesis, $G^{\prime}$ has an interval coloring $\alpha$
with no more than four colors. Since $G$ is $2$-connected, we have
$d_{G}(x)=d_{G}(y)=3$.

If $S(x,\alpha)=S(y,\alpha)=\{c,c+1\}$, then if $c=1$, then we color
the edges $ux$ and $vy$ with color $3$ and the edge $uv$ with the
color $2$; otherwise we color the edges $ux$ and  $vy$ with color
$c-1$ and the edge $uv$ with color $c$. If $S(x,\alpha)\cup
S(y,\alpha)=\{c,c+1,c+2\}$, then without loss of generality we may
assume that $S(x,\alpha)=\{c,c+1\}$ and $S(y,\alpha)=\{c+1,c+2\}$.
We color the edge $ux$ with color $c+2$, the edge $uv$ with color
$c+1$ and the edge $vy$ with color $c$. Clearly, the obtained
coloring is an interval coloring of the graph $G$ with no more than
four colors.

Case 2: $uv,vw,uw\in E(G)$ and $d_{G}(u)=d_{G}(w)=3$ and
$d_{G}(v)=2$.

In this case by contracting the $uvw$ triangle to a single vertex
$v^{\star}$, we obtain a $2$-connected outerplanar graph
$G^{\prime}$ with $\Delta(G^{\prime})\leq 3$. If $G^{\prime}$ is an
odd cycle, then we color the edge $uw$ with color $3$, the first
edge of the $u,w$-path in $G$ with color $4$ and the remaining edges
alternately with colors $3$ and $2$. Next, we color the edge $uv$
with color $2$ and the edge $vw$ with color $1$. Clearly, the
obtained coloring is an interval $4$-coloring of the graph $G$.

Now we can suppose that $G^{\prime}$ is a $2$-connected outerplanar
graph with $\Delta(G^{\prime})\leq 3$ which is not an odd cycle. By
induction hypothesis, $G^{\prime}$ has an interval coloring $\alpha$
with no more than four colors.

Let $S(v^{\star},\alpha)=\{c,c+1\}$. If $c=1$, then we color the
edge $uw$ with color $3$ and the edges $uv$ and $vw$ with colors $1$
and $2$ or $2$ and $1$ depending on the colors of the colored edges
incident to vertices $u$ and $w$; otherwise we color the edge $uw$
with color $c-1$ and the edges $uv$ and $vw$ with colors $c$ and
$c+1$ or $c+1$ and $c$ depending on the colors of the colored edges
incident to vertices $u$ and $w$. Clearly, the obtained coloring is
an interval coloring of the graph $G$ with no more than four colors.
~$\square$
\end{proof}

\begin{figure}[h]
\begin{center}
\includegraphics[width=15pc]{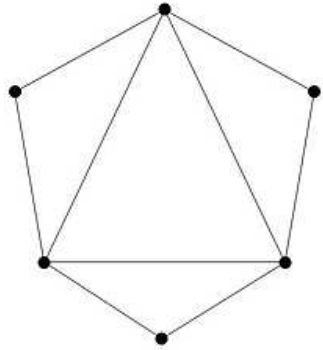}\\
\caption{A 2-connected outerplanar graph $G$ with $\Delta(G)=4$
which has no interval coloring.}\label{fig1}
\end{center}
\end{figure}

Next we prove the following result:

\begin{theorem}
\label{mytheorem2} If $G$ is a $2$-connected outerplanar graph $G$
with $\Delta(G)=3$, then $G\in \mathfrak{N}$ and
\begin{center}
$w(G)=\left\{
\begin{tabular}{ll}
$3$, & if $\vert V(G)\vert$ is even, \\
$4$, & if $\vert V(G)\vert$ is odd. \\
\end{tabular}%
\right.$
\end{center}
\end{theorem}
\begin{proof}
By Theorem \ref{mytheorem1}, we have $G\in \mathfrak{N}$ and
$w(G)\leq 4$. On the other hand, since $\Delta(G)=3$, we obtain
$3\leq w(G)\leq 4$. We consider two cases.

Case 1: $\vert V(G)\vert$ is even.

Since $G$ is $2$-connected and outerplanar, it is clear that $G$ is
Hamiltonian. Let $C$ be a Hamiltonian cycle in $G$. Since $\vert
V(G)\vert$ is even, clearly $\vert C\vert$ is even, too. Now we
construct an interval $3$-coloring of the graph $G$. First we color
the edges of the cycle $C$ alternately with colors $1$ and $2$. Next
we color the edges from the set $E(G)\setminus E(C)$ by color $3$.
Since $E(G)\setminus E(C)$ is a matching in $G$, the obtained
coloring is an interval $3$-coloring of the graph $G$ and thus
$w(G)=3$.

Case 2: $\vert V(G)\vert$ is odd.

For the proof, it suffice to show that $w(G)\geq 4$.

Suppose, to the contrary, that $G$ has an interval $3$-coloring
$\alpha$. In this case we consider the set $S(v,\alpha)$ for every
$v\in V(G)$. Since $G$ is $2$-connected and outerplanar, we have
$1\leq {\min}~S(v,\alpha)\leq 2$ for every $v\in V(G)$. This implies
that the edges with color $2$ form a perfect matching in $G$, but
this contradicts the fact that $\vert V(G)\vert$ is odd. Thus,
$w(G)=4$. ~$\square$
\end{proof}

On the other hand, there are $2$-connected outerplanar graphs $G$
with $\Delta(G)=4$ which are not interval colorable. For example,
the graph $G$ shown in Fig. \ref{fig1} has no interval coloring. Now
we show a more general result. For that we define a triangle graph
$T_{k,l,m}$ ($k,l,m\in \mathbb{N}$) as follows:

\begin{center}
$V\left(T_{k,l,m}\right)=\{x,y,z,u_{1},\ldots,
u_{2k-1},v_{1},\ldots,v_{2l-1},w_{1},\ldots,w_{2m-1}\}$ and\
$E\left(T_{k,l,m}\right)=\{xy,xu_{1},u_{2k-1}y,yz,yv_{1},v_{2l-1}z,xz,xw_{1},w_{2m-1}z\}\cup$\\
$\{u_{i}u_{i+1}:1\leq i\leq 2k-2\}\cup \{v_{i}v_{i+1}:1\leq i\leq
2l-2\}\cup \{w_{i}w_{i+1}:1\leq i\leq 2m-2\}$.
\end{center}

Clearly, $T_{k,l,m}$ is a $2$-connected outerplanar graph with
$\Delta\left(T_{k,l,m}\right)=4$.

\begin{theorem}
\label{mytheorem3} For any $k,l,m\in \mathbb{N}$, we have
$T_{k,l,m}\notin \mathfrak{N}$.
\end{theorem}
\begin{proof}
Suppose, to the contrary, that the graph $T_{k,l,m}$ has an interval
$t$-coloring $\alpha$ for some $t\geq 4$.

Since all degrees of vertices of the graph $T_{k,l,m}$ are even, the
sets $S(x,\alpha),S(y,\alpha)$ and $S(z,\alpha)$ contain two even
colors and two odd colors, and the sets
$S(u_{i},\alpha),S(v_{j},\alpha)$ and $S(w_{p},\alpha)$ contain one
even color and one odd color for
$i=1,\ldots,2k-1,j=1,\ldots,2l-1,p=1,\ldots,2m-1$.

Consider the triangle $xyz$. Clearly, there is a vertex of the
triangle for which the colors of two incident edges of the triangle
have the same parity. Without loss of generality, we may assume that
this vertex is $x$ and $\alpha(xy)$ and $\alpha(xz)$ have the same
parity. If $\alpha(xy)$ and $\alpha(xz)$ are even colors, then
$\alpha(xu_{1})$ and $\alpha(xw_{1})$ are odd colors, and thus
$\alpha(u_{2k-1}y)$ and $\alpha(w_{2m-1}z)$ are even colors. This
implies that $\alpha(yz)$, $\alpha(yv_{1})$ and $\alpha(v_{2l-1}z)$
are odd colors. On the other hand, since $\alpha(yz)$,
$\alpha(yv_{1})$ are odd colors, we obtain $\alpha(v_{2l-1}z)$ is an
even color, which is a contradiction. Similarly, if $\alpha(xy)$ and
$\alpha(xz)$ are odd colors, then $\alpha(xu_{1})$ and
$\alpha(xw_{1})$ are even colors, and thus $\alpha(u_{2k-1}y)$ and
$\alpha(w_{2m-1}z)$ are odd colors. This implies that $\alpha(yz)$,
$\alpha(yv_{1})$ and $\alpha(v_{2l-1}z)$ are even colors. On the
other hand, since $\alpha(yz)$, $\alpha(yv_{1})$ are even colors, we
obtain $\alpha(v_{2l-1}z)$ is an odd color, which is a
contradiction. ~$\square$
\end{proof}

\bigskip

\section{Interval edge-colorings of outerplanar triangulations}\

In \cite{b3}, Axenovich showed that all outerplanar triangulations
with more than three vertices and without separating triangles are
interval colorable, and also she posed the following

\begin{question}
Is it true that an outerplanar triangulation has an interval
coloring if and only if it does not have a separating triangle?
\end{question}

\begin{figure}[h]
\begin{center}
\includegraphics[width=30pc]{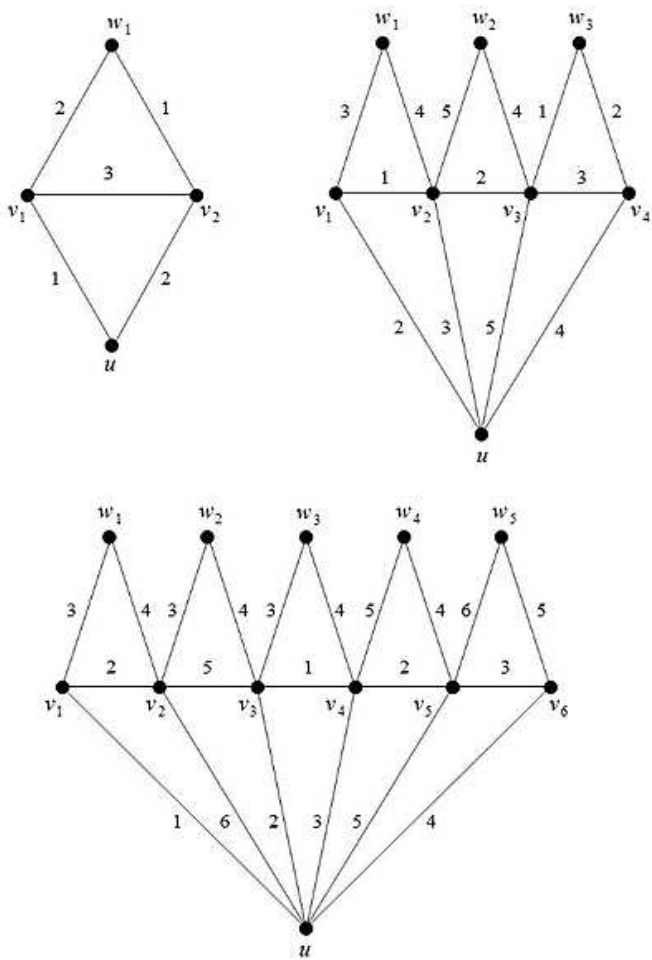}\\
\caption{An interval $3$-coloring of $TF_{3}$, an interval
$5$-coloring of $TF_{5}$ and an interval $6$-coloring of
$TF_{7}$.}\label{fig2}
\end{center}
\end{figure}

In this section we give a negative answer to the question. For that
we define a triangular fan graph $TF_{n}$ ($n\geq 3$) as follows:
\begin{center}
$V(TF_{n})=\{u,v_{1},\ldots,v_{n-1},w_{1},\ldots,w_{n-2}\}$ and\
$E(TF_{n})=\{uv_{i}:1\leq i\leq n-1\}\cup
\{v_{i}w_{i},w_{i}v_{i+1},v_{i}v_{i+1}:1\leq i\leq n-2\}$.
\end{center}

Clearly, $TF_{n}$ is an outerplanar triangulation.

\begin{theorem}
\label{mytheorem4} For any $n\geq 3$, $TF_{n}$ has an interval
$\Delta(TF_{n})$-coloring.
\end{theorem}
\begin{proof}
We consider two cases.

\begin{figure}[h]
\begin{center}
\includegraphics[width=30pc]{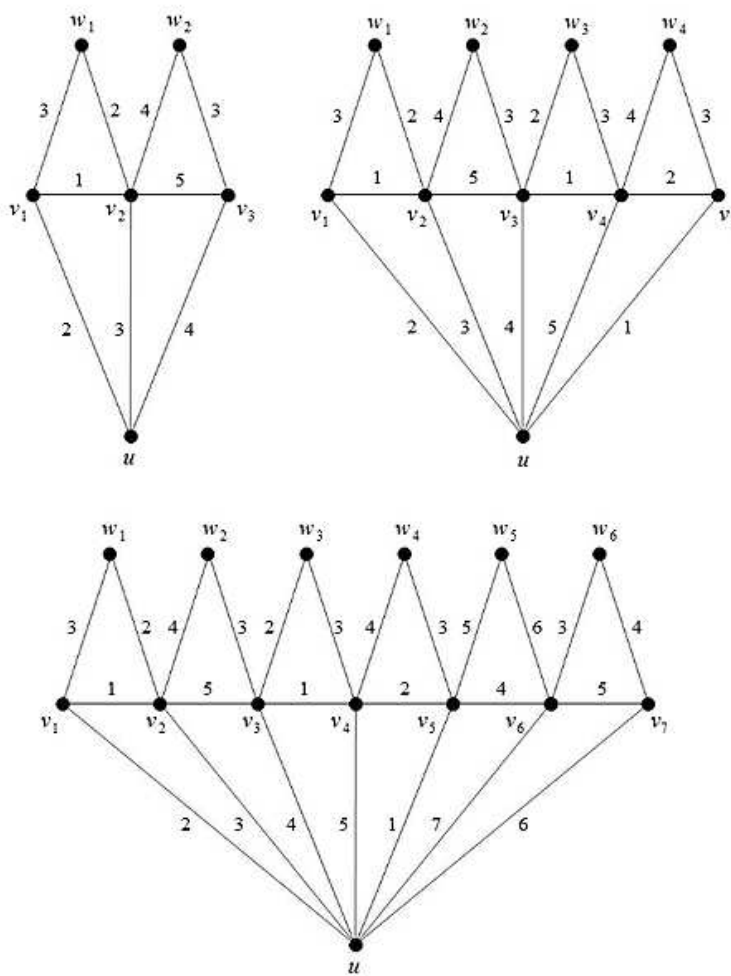}\\
\caption{Interval $5$-colorings of $TF_{4}$ and $TF_{6}$, and an
interval $7$-coloring of $TF_{8}$.}\label{fig3}
\end{center}
\end{figure}

Case 1: $n$ is odd.

Fig. \ref{fig2} gives an interval $3$-coloring of $TF_{3}$, an
interval $5$-coloring of $TF_{5}$ and an interval $6$-coloring of
$TF_{7}$. Let $\alpha$ be an interval $6$-coloring of $TF_{7}$ shown
in Fig. \ref{fig2}. Now we define an edge-coloring $\beta$ of the
graph $TF_{n}$ as follows:
\begin{description}
\item[(1)] for every $e\in E(TF_{7})$, let $\beta(e)=\alpha(e)$;

\item[(2)] for $i=3,\ldots,\frac{n-3}{2}$, let\
\begin{center}
$\beta(uv_{2i+1})=2i+2$ and $\beta(uv_{2i+2})=2i+1$,
$\beta(v_{2i+1}w_{2i})=\beta(w_{2i+1}v_{2i+2})=2i-1$,\
$\beta(v_{2i+1}v_{2i+2})=\beta(v_{2i}w_{2i})=2i$,
$\beta(v_{2i}v_{2i+1})=2i+1$ and $\beta(v_{2i+1}w_{2i+1})=2i-2$.
\end{center}
\end{description}

It is not difficult to see that $\beta$ is an interval
$\Delta(TF_{n})$-coloring of $TF_{n}$ for odd $n$.

Case 2: $n$ is even.

Fig. \ref{fig3} gives interval $5$-colorings of $TF_{4}$ and
$TF_{6}$, and an interval $7$-coloring of $TF_{8}$. Let $\alpha$ be
an interval $7$-coloring of $TF_{8}$ shown in Fig. \ref{fig3}. Now
we define an edge-coloring $\beta$ of the graph $TF_{n}$ as follows:
\begin{description}
\item[(1)] for every $e\in E(TF_{8})$, let $\beta(e)=\alpha(e)$;

\item[(2)] for $i=3,\ldots,\frac{n-4}{2}$, let\
\begin{center}
$\beta(uv_{2i+2})=2i+3$ and $\beta(uv_{2i+3})=2i+2$,
$\beta(v_{2i+2}w_{2i+1})=\beta(w_{2i+2}v_{2i+3})=2i$,
$\beta(v_{2i+2}v_{2i+3})=\beta(v_{2i+1}w_{2i+1})=2i+1$,\\
$\beta(v_{2i+1}v_{2i+2})=2i+2$ and $\beta(v_{2i+2}w_{2i+2})=2i-1$.
\end{center}
\end{description}

It is not difficult to see that $\beta$ is an interval
$\Delta(TF_{n})$-coloring of $TF_{n}$ for even $n$. ~$\square$
\end{proof}

\begin{corollary}
\label{mycorollary} For any $n\geq 3$, we have $TF_{n}\in
\mathfrak{N}$ and $w(TF_{n})=\Delta(TF_{n})$.
\end{corollary}

Clearly, $uv_{i}v_{i+1}$ is a separating triangle in $TF_{n}$ for
$i=2,\ldots,n-3$. Thus, for $n\geq 5$, $TF_{n}$ has $n-4$ separating
triangles.


\bigskip


\begin{thebibliography}{99}

\bibitem{b1} A.S. Asratian, R.R. Kamalian, Interval colorings of edges of a
multigraph, Appl. Math. 5 (1987), 25-34 (in Russian).

\bibitem{b2} A.S. Asratian, R.R. Kamalian, Investigation on interval
edge-colorings of graphs, J. Combin. Theory Ser. B 62 (1994),
34-43.

\bibitem{b3} M.A. Axenovich, On interval colorings of planar graphs, Congr.
Numer. 159 (2002), 77-94.

\bibitem{b4} S. Fiorini, On the chromatic index of outerplanar
graphs, J. Combin. Theory Ser. B 18 (1975), 35-38.

\bibitem{b5} K. Giaro, M. Kubale, Compact scheduling of zero-one time
operations in multi-stage systems, Discrete Appl. Math. 145 (2004),
95-103.

\bibitem{b6} R.R. Kamalian, Interval colorings of complete bipartite graphs
and trees, preprint, Comp. Cen. of Acad. Sci. of Armenian SSR,
Erevan, 1989 (in Russian).

\bibitem{b7} P.A. Petrosyan, Interval edge-colorings of complete graphs and
$n$-dimensional cubes, Discrete Math. 310 (2010), 1580-1587.

\bibitem{b8} P.A. Petrosyan, H.H. Khachatrian, H.G. Tananyan, Interval
edge-colorings of Cartesian products of graphs I, Discuss. Math.
Graph Theory, 2013, to appear.

\bibitem{b9} S.V. Sevast'janov, Interval colorability of the edges of a
bipartite graph, Metody Diskret. Analiza 50 (1990), 61-72 (in
Russian).

\bibitem{b10} D.B. West, Introduction to Graph Theory, Prentice-Hall, New
Jersey, 2001.

\end{thebibliography}
\end{document}